\documentclass[12pt,british]{article}
\usepackage[latin9]{inputenc}
\usepackage{babel}
\usepackage{refstyle}
\usepackage{mathtools}
\usepackage{amsmath}
\usepackage{amsthm}
\usepackage{amssymb}
\usepackage{wasysym}
\usepackage[all]{xy}
\usepackage[unicode=true,pdfusetitle,
 bookmarks=true,bookmarksnumbered=false,bookmarksopen=false,
 breaklinks=false,pdfborder={0 0 0},pdfborderstyle={},backref=false,colorlinks=false]
 {hyperref}

\makeatletter

\AtBeginDocument{\providecommand\lemref[1]{\ref{lem:#1}}}
\AtBeginDocument{\providecommand\propref[1]{\ref{prop:#1}}}
\RS@ifundefined{subsecref}
  {\newref{subsec}{name = \RSsectxt}}
  {}
\RS@ifundefined{thmref}
  {\def\RSthmtxt{theorem~}\newref{thm}{name = \RSthmtxt}}
  {}
\RS@ifundefined{lemref}
  {\def\RSlemtxt{lemma~}\newref{lem}{name = \RSlemtxt}}
  {}

\theoremstyle{plain}
\newtheorem{thm}{\protect\theoremname}[section]
  \theoremstyle{plain}
  \newtheorem{lem}[thm]{\protect\lemmaname}
  \theoremstyle{definition}
  
  \theoremstyle{remark}
  \newtheorem*{claim*}{\protect\claimname}
  \theoremstyle{plain}
  \newtheorem{cor}[thm]{\protect\corollaryname}
  \theoremstyle{plain}
  \newtheorem{prop}[thm]{\protect\propositionname}
  \theoremstyle{definition}
  \newtheorem{example}[thm]{\protect\examplename}
  \theoremstyle{remark}
  \newtheorem{rem}[thm]{\protect\remarkname}

\usepackage{fullpage}
\date{}

\newref{lem}{refcmd={Lemma \ref{#1}}}
\newref{fact}{refcmd={Fact \ref{#1}}}
\newref{prop}{refcmd={Proposition \ref{#1}}}
\newref{cor}{refcmd={Cor \ref{#1}}}

\AtBeginDocument{
  
}

\makeatother

  \providecommand{\claimname}{Claim}
  \providecommand{\corollaryname}{Corollary}
  \providecommand{\definitionname}{Definition}
  \providecommand{\examplename}{Example}
  \providecommand{\lemmaname}{Lemma}
  \providecommand{\propositionname}{Proposition}
  \providecommand{\remarkname}{Remark}
\providecommand{\theoremname}{Theorem}

\begin{document}
\global\long\def\invlim{\varprojlim}

\global\long\def\bdu{\bigcupdot}

\global\long\def\du{\cupdot}

\global\long\def\sm{\!\setminus\!}

\global\long\def\res#1{\!\restriction_{#1}}

\global\long\def\explain#1#2{\underset{\underset{\mathclap{#2}}{\downarrow}}{#1}}

\global\long\def\conn#1{#1^{0}}

\global\long\def\Sp{\mathrm{Span}}

\global\long\def\trnobrackets{\mathrm{d}}
\global\long\def\tr#1{\trnobrackets(#1)}

\global\long\def\srnobrackets{\mathrm{sr}}
\global\long\def\sr#1{\srnobrackets(#1)}

\global\long\def\adp{\Yup}

\global\long\def\gen#1#2{\mathrm{Gen}_{#2}(#1)}

\global\long\def\gas#1{\text{\lightning}(#1)}

\global\long\def\cb#1#2{\bar{B}_{#1}(#2)}

\global\long\def\ob#1#2{B_{#1}(#2)}

\title{Gasch{\"u}tz Lemma for Compact Groups}
\author{Tal Cohen and Tsachik Gelander\thanks{The authors acknowledge the support of the ISF-Moked grant 2095/15.}}

\maketitle
\begin{abstract}
We prove the Gasch{\"u}tz Lemma holds for all metrisable compact
groups.
\end{abstract}

\section{Introduction}

For a group $G$, denote by $\tr G$ its minimal number of generators.
A classic result by Gasch{\"u}tz asserts the following:
\begin{lem}[The Gasch{\"u}tz Lemma \cite{Gas}]
Let $G$ be a finite group and $f:G\to H$ an epimorphism. If $H=\left<h_{1},\dots,h_{n}\right>$
for $n\geqslant\tr G$, then there are $g_{1},\dots,g_{n}\in G$ satisfying
$f\left(g_{i}\right)=h_{i}$ for $i=1,\dots,n$ and $G=\left<g_{1},\dots,g_{n}\right>$.
\end{lem}
It is known and straightforward that the Gasch{\"u}tz Lemma holds for profinite groups as well,
where naturally one replaces \emph{generation} by topological
generation and assumes the epimorphism is continuous. However, it does
not hold for arbitrary infinite groups. For example, even for $\mathbb{Z}\to\mathbb{Z}/5\mathbb{Z}$
the generator $2+5\mathbb{Z}$ doesn't admit a generating lift, even
though $\tr{\mathbb{Z}}=1$. 

We will prove the lemma does hold for first-countable (i.e.~metrisable) compact groups.
Moreover, we will give an example that shows the lemma does not hold
for general non-metrisable compact groups.

We assume all groups are Hausdorff, all homomorphisms are continuous and all subgroups are closed. We make the following definitions:

\begin{enumerate}
\item If $g_{1},\dots,g_{n}$ are elements of a topological group $G$,
we say they are \emph{generators }(rather than topological generators) of $G$ if $\overline{\left\langle g_{1},\dots,g_{n}\right\rangle }=G$.
We denote by $\tr G$ the minimal number of generators of the group
$G$. 
\item An epimorphism $f:G\to H$ is \emph{Gasch{\"u}tz }if for every $h_{1},\dots,h_{n}\in H$
with $n\geqslant\tr G$ which generate $H$ there are lifts $g_{1},\dots,g_{n}$
($f(g_{i})=h_{i}$) which generate $G$. It is \emph{super-Gasch{\"u}tz}
if $\ker f$ is locally compact and
\[
\{ (k_{1},\dots,k_{n})\in(\ker f)^{n}|\overline{\left\langle g_{1}k_{1},\dots,g_{n}k_{n}\right\rangle }=G\} 
\]
is of full Haar measure in $(\ker f)^{n}$
(so that ``almost all lifts'' of $h_{1},\dots,h_{n}$ generate $G$).
\item $G$ is \emph{Gasch{\"u}tz }if every open epimorphism $f:G\to H$ is
Gasch{\"u}tz.
\item $G$ is \emph{connectedly Gasch{\"u}tz} if every open epimorphism $f:G\to H$
with a connected kernel is Gasch{\"u}tz.
\item $G$ is \emph{connectedly super-Gasch{\"u}tz} if it is locally compact and every open epimorphism
$f:G\to H$ with a connected kernel is super-Gasch{\"u}tz.
\end{enumerate}

\medskip

In this terminology, the Gasch{\"u}tz Lemma says all profinite groups
are Gasch{\"u}tz. Observe that the restriction to open epimorphisms
is meaningless for compact groups, since all epimorphisms are open.

We now state our main theorem:
\begin{thm}[Main Theorem]
\label{thm:main}Every metrisable compact group is Gasch{\"u}tz and connectedly
super-Gasch{\"u}tz.
\end{thm}

\section{Main Theorem}

First, let us explain why being connectedly Gasch{\"u}tz implies being Gasch{\"u}tz. The following lemma is proved in the exact same
way as the original Gasch{\"u}tz Lemma, but we repeat the proof for the
convenience of the reader.
\begin{lem}
\label{lem:fin-ker}Suppose $G$ is a compact group and that $f:G\to H$
is an epimorphism with a finite kernel. Then $f$ is Gasch{\"u}tz.
\end{lem}
\begin{proof}
Let $G$ be a compact group and $f:G\to H$ an epimorphism with a
finite kernel. For a generating system $\underline{h}\in H^{n}$ and
a subgroup $F\leqslant G$, denote by $\varphi_{n}^{F}(\underline{h})$
the (finite) number of lifts of $\underline{h}$ to $F$ which generate $F$.
That is, 
\[
\varphi_{n}^{F}(\underline{h})=\#\{\underline{g}\in F^{n}|\overline{\langle\underline{g}\rangle}=F,\underline{f}(\underline{g})=\underline{h}\}.
\]
\begin{claim*}
For every $F\leqslant G$ and $n\geqslant\tr G$ the map $\varphi_{n}^{F}$
is constant.
\end{claim*}
We fix some $n\geqslant\tr G$ and prove the claim by induction on
$\left|F\cap\ker f\right|$, which is finite by assumption. The case
$F\cap\ker f=\left\{ 1\right\} $ is trivial: $\varphi_{n}^{F}(\underline{h})=1$
regardless of $\underline{h}$, since there is only one lift of $\underline{h}$
and it generates $F$ (since $f$ is an isomorphism). 

Now let $F\leqslant G$ and assume $\varphi_{n}^{E}$ is constant
for all subgroups $E\leqslant G$ with $|E\cap\ker f|<\left|F\cap\ker f\right|$.
Let $\underline{h}\in H^{n}$ be a generating system of $H$; we need
to show $\varphi_{n}^{F}(\underline{h})$ only depends on $F$ and
$n$. Consider the (possibly empty) set $\mathcal{E}$ of all
proper subgroups of $F$ which project onto $H$, 
\[
\mathcal{E}\coloneqq\left\{ E<F\;\middle|\;f\left(E\right)=H\right\} .
\]
Observe there are exactly $\left|F\cap\ker f\right|^{n}$ possible
lifts of $\underline{h}$ to $F$, and clearly each one of them generates
a subgroup which projects onto $H$ (since the projection contains
$\underline{h}$). Therefore, if a lift of $\underline{h}$ to $F$
does not generate $F$, it must generate a subgroup belonging to $\mathcal{E}$;
this implies 
\[
\varphi_{n}^{F}(\underline{h})=\left|F\cap\ker f\right|^{n}-\sum_{E\in\mathcal{E}}|\varphi_{n}^{E}(\underline{h})|.
\]
By the induction hypothesis $\varphi_{n}^{E}$ is constant, so the
right hand side depends only on $F$ and $n$, as needed.

The lemma easily follows: we know there is some generating system
$\underline{g}\in G^{n}$ of $G$, which means $\underline{f}(\underline{g})$
is a generating system of $H$ such that $\varphi_{n}^{G}(\underline{f}(\underline{g}))>0$,
and thus $\varphi_{n}^{G}(\underline{h})>0$ for every generating
system $\underline{h}\in H^{n}$.
\end{proof}
\begin{cor}
\label{cor:profin-ker}Suppose $G$ is a compact group and that $f:G\to H$
is an epimorphism with a profinite kernel. Then $f$ is Gasch{\"u}tz.
\end{cor}
\begin{proof}
Denote $K=\ker f$, where $K$ is profinite. Using the Peter-Weyl
theorem one can see $G$ is an inverse limit of Lie groups, $G=\invlim\left(G_{\alpha}\right)_{\alpha\in A}$.\footnote{A convenient reference for the metric case (which is all we will eventually
need) is \cite[4.7]{MZ}.} Moreover, $H\cong\invlim\left(G_{\alpha}/K_{\alpha}\right)_{\alpha\in A}$
for $K_{\alpha}\coloneqq p_{\alpha}\left(K\right)$ (where $p_{\alpha}:G\to G_{\alpha}$
is the projection) and modulo this isomorphism $f:G\to H$ is given
by
\[
\left(g_{\alpha}\right)_{\alpha\in A}\mapsto\left(g_{\alpha}K_{\alpha}\right)_{\alpha\in A}.
\]
The groups $K_{\alpha}=p_{\alpha}\left(K\right)$ are profinite subgroups
of Lie groups and hence are finite. Therefore by \lemref{fin-ker}
the maps $G_{\alpha}\to G_{\alpha}/K_{\alpha}$ are Gasch{\"u}tz.

Let $\underline{h}\in\left(G/K\right)^{n}$ be a generating system
of $G/K$ with $n\geqslant\tr G$, and pick an arbitrary lift $\underline{g}\in G^{n}$.
Denoting by $\pi_{\alpha}:G/K\to G_{\alpha}/K_{\alpha}$ the quotient
map, we get the following commutative diagram:
\[
\xymatrix{G\ar[r]^{{p_{\alpha}}}\ar[d]_{f} & G_{\alpha}\ar[d]^{f_{\alpha}}\\
G/K\ar[r]_{\pi_{\alpha}} & G_{\alpha}/K_{\alpha}
}
\]
Therefore, $\underline{p}_{\alpha}(\underline{g})$ is a lift of $\underline{\pi}_{\alpha}\left(\underline{h}\right)$
via $G_{\alpha}\to G_{\alpha}/K_{\alpha}$. Thus, there is some $\underline{k}\in{K_{\alpha}}^{n}$
such that $\underline{p}_{\alpha}(\underline{g})\cdot\underline{k}$
generates $G_{\alpha}$. In other words, the set
\[
\tilde{\mathcal{L}}_{\alpha}\coloneqq\left\{ \underline{k}\in{K_{\alpha}}^{n}\middle|\overline{\langle\underline{p}_{\alpha}(\underline{g})\underline{k}\rangle}=G_{\alpha}\right\} 
\]
is nonempty for every $\alpha\in A$, and hence 
\[
\mathcal{L}_{\alpha}\coloneqq (\underline{p}_{\alpha}\res{K^{n}})^{-1}(\tilde{\mathcal{L}}_{\alpha})
\]
is nonempty. It is moreover closed as an inverse image of a finite
set. We show the collection $\{\mathcal{L}_{\alpha}\}_{\alpha\in A}$
satisfies the finite intersection property, which means it has a
nonempty intersection.

If $S\subseteq A$ is finite, take $\beta\geqslant S$ and some $\underline{k}\in\mathcal{L}_{\beta}$.
Then $\underline{p}_{\beta}(\underline{k})\in\tilde{\mathcal{L}}_{\beta}$, which
means $\overline{\langle \underline{p}_{\beta}(\underline{g}\cdot\underline{k})\rangle}=G_{\beta}$,
so also
\[
\overline{\langle \underline{p}_{\alpha}(\underline{g}\cdot\underline{k})\rangle}=\overline{\langle \underline{p}_{\beta\alpha}(\underline{p}_{\beta}(\underline{g}\cdot\underline{k}))\rangle}=G_{\alpha},
\]
i.e.~$\underline{k}\in\mathcal{L}_{\alpha}$ for every $\alpha\in S$.
So $\bigcap_{\alpha\in S}\mathcal{L}_{\alpha}\neq\varnothing$, as
needed. Therefore the set
\[
\mathcal{L}\coloneqq\bigcap_{\alpha\in A}\mathcal{L}_{\alpha}
\]
is also nonempty. Pick $\underline{k}\in\mathcal{L}$; by definition
$\underline{p}_{\alpha}(\underline{g}\cdot\underline{k})$ generates $G_{\alpha}$
for every $\alpha\in A$, so $\underline{g}\cdot\underline{k}$ generates
$G$.
\end{proof}
\begin{cor}
A compact group is Gasch{\"u}tz if and only if it is connectedly Gasch{\"u}tz.
\end{cor}
\begin{proof}
Obviously if a group is Gasch{\"u}tz it is connectedly Gasch{\"u}tz.
Suppose now $G$ is compact and connectedly Gasch{\"u}tz, and let $f:G\to H$
be some epimorphism. Denote $K=\ker f$; by the third isomorphism
theorem, the diagram
\[
\xymatrix{G\ar[r]\ar[d] & G/K^{\circ}\ar[d]\\
G/K\ar[r]\sp(0.3)\sim & \left(G/K^{\circ}\right)/\left(K/K^{\circ}\right)
}
\]
commutes. Thus $G\to G/K$ is the composition $G\to G/K^{\circ}\to\left(G/K^{\circ}\right)/\left(K/K^{\circ}\right)$.
The first map is Gasch{\"u}tz since $G$ is connectedly Gasch{\"u}tz,
and the second map is Gasch{\"u}tz since it has a profinite kernel (Corollary \ref{cor:profin-ker}),
so $G\to H$ is also Gasch{\"u}tz.
\end{proof}
We will use the following well known fact. Since we couldn't find
a convenient reference for its proof, we include it here.
\begin{prop}
\label{prop:conjugacy}A compact Lie group has only countably many conjugacy
classes of subgroups.\footnote{Recall that by ``subgroups" we mean closed subgroups.}
\end{prop}
\begin{proof}
Let $G$ be a compact Lie group. Obviously, it is enough to prove that
for each $n$ there are countably many conjugacy classes of
subgroups of dimension $n$. Fix some $n$.

Recall that the space $\mathrm{Sub}(G)$ of subgroups of $G$, equipped
with the Chabauty topology, is a compact metric space (see \cite{Bir,SubG}).
It is in particular second countable, and therefore the subset $\mathrm{Sub}_{n}(G)$
of $n$-dimensional subgroups is also second countable. Thus,
every open cover of $\mathrm{Sub}_{n}(G)$ has a countable subcover.

By \cite[Theorem 5.3]{MZ}, for every
$n$-dimensional subgroup $F\leqslant G$ there is an open subset
$U_{F}\subseteq G$ such that $F\subseteq U_{F}$ and such that every
subgroup of $G$ contained in $U_{F}$ can be conjugated to
a subgroup of $F$. For each such $F$, let $\mathcal{U}_{F}\subseteq\mathrm{Sub}_{n}(G)$
be the open set of subgroups contained in $U_{F}$. Since $\{\mathcal{U}_{F}\}$
is an open cover of $\mathrm{Sub}_{n}(G)$, there is a countable
subcover $\{\mathcal{U}_{F_{i}}\}_{i\in\mathbb{N}}$. Moreover, since
each $F_{i}$ is an $n$-dimensional compact Lie group, it has only
finitely many subgroups of dimension $n$. Since every $n$-dimensional
subgroup of $G$ is conjugated to a (necessarily $n$-dimensional
subgroup) of $F_{i}$ for some $i\in\mathbb{N}$, it follows $G$
has only countably many conjugacy classes of $n$-dimensional subgroups.
\end{proof}

The following result is a folklore, yet we will provide a proof for completeness. 

\begin{lem}
If $G$ is a Lie group and $P\leqslant G$ is a subgroup, then there
is an immersed submanifold $A\subseteq G$ such that $PA=G$ and $\dim P+\dim A=\dim G$.
\end{lem}
\begin{proof}
Denote by $\pi:G\to G/P$ the quotient map. For every $gP\in G/P$
take an open neighbourhood $U_{gP}\ni gP$ such that $\pi^{-1}(U_{gP})$
is diffeomorphic to $U_{gP}\times P$ and such that modulo this diffeomorphism
$\pi$ is given by the projection onto the first coordinate. Since
$\{U_{gP}\}_{gP\in G/P}$ is an open cover of the manifold $G/P$,
we can take a countable subcover $\left\{ U_{i}\right\} _{i\in\mathbb{N}}$.
By construction, for every $i\in\mathbb{N}$ we can take a submanifold
$\hat{U}_{i}\subseteq G$ of dimension $\dim G-\dim P$ such that
$\pi\res{\hat{U}_{i}}:\hat{U}_{i}\to U_{i}$ is a diffeomorphism.
Therefore, the countable union $A\coloneqq\bigcup_{i\in\mathbb{N}}\hat{U}_{i}$
is an immersed submanifold of dimension $\dim G-\dim P$ such that
$\pi(A)=G/P$, and therefore $PA=G$.
\end{proof}

We can now complete the proof of the main theorem.

\begin{proof}[Proof of the Main Theorem]
We will prove all first-countable compact groups are connectedly super-Gasch{\"u}tz, and in particular Gasch{\"u}tz. First we will assume $G$ is a compact Lie group, and then the general
case will follow by a standard argument of inverse limits.

Let $G,H$ be compact Lie groups, $f:G\to H$ an epimorphism with a
connected kernel $K$, and $h_{1},\dots,h_{n}\in H$ generators of
$H$ with $n\geqslant\tr G$. 

First we consider the special case $n=1$. Observe that in this case
$G$ must be abelian. Pick some lift $g$ of $h_1$ and denote by
$\mathcal{F}$ the set of proper subgroups of $G$ projecting onto
$H$. Now, we need to show the set of $k\in K$
such that $gk$ doesn't generate $G$ is null.
First observe that $K\cap F$ is null in $K$ for every
$F\in\mathcal{F}$ since $K$ is connected. For each $F\in\mathcal{F}$
pick a lift $r_F\in F$ of $h_1$, and
observe that the set $\{k\in K|\overline{\langle r_F k\rangle}=F\}$
is null since it is contained in $K\cap F$. Since $\mathcal{F}$ is countable (e.g.~by \propref{conjugacy}), we see by left invariance
that
\begin{align*}
\mu\left(\left\{ k\in K\middle|\left\langle {g}{k}\right\rangle \neq G\right\}\right) & \leqslant\sum_{F\in\mathcal{F}}\mu\left(\left\{ {k}\in K\middle|\left\langle {g}{k}\right\rangle =F\right\}\right) \\
 & =\sum_{F\in\mathcal{F}}\mu\left(\left\{ {k}\in K\middle|\left\langle {r}_{F}{k}\right\rangle =F\right\}\right)=0
\end{align*}
(where $\mu$ is the Haar measure of $K$), as needed.

We can now assume $n\geqslant2$. For every subgroup $P\leqslant G$
such that $f(P)=H$, denote 
\[
K_{P}=K\cap P=\ker(f\res P),
\]
and for $i=1,\dots,n$ denote 
\[
\mathcal{L}_{P}^{i}=\left\{ p\in P\middle|f(p)=h_{i}\right\} .
\]
Observe $\mathcal{L}_{P}^{i}$ is a shift of $K_{P}$ (by an arbitrary
lift of $h_{i}$ to $P$), and hence has a well defined measure (induced
from the Haar measure of $K_{P}$). It is moreover a submanifold of
dimension $\dim K_{P}$, which is strictly smaller than $\dim K$
if $P<G$ (since $K$ is connected).

We need to show the subset of tuples in $\prod_{i=1}^{n}\mathcal{L}_{G}^{i}$
generating a proper subgroup of $G$ is null. By \propref{conjugacy},
we can take a countable set $\mathcal{P}$ containing a representative
of the conjugacy class of each proper subgroup $P<G$. Thus a tuple
in $\prod_{i=1}^{n}\mathcal{L}_{G}^{i}$ not generating $G$ must
generate a subgroup of the form $xPx^{-1}$ for some $P\in\mathcal{P}$
and some $x\in G$. That is,
\[
\left\{ \underline{g}\in\prod_{i=1}^{n}\mathcal{L}_{G}^{i}\middle|\overline{\left\langle \underline{g}\right\rangle }\neq G\right\} =\bigcup_{P\in\mathcal{P}}\left\{ \underline{g}\in\prod_{i=1}^{n}\mathcal{L}_{G}^{i}\middle|\overline{\left\langle \underline{g}\right\rangle }=xPx^{-1}\text{ for some }x\in G\right\} .
\]
Since this is a countable union, it is enough to show that 
\[
\left\{ \underline{g}\in\prod_{i=1}^{n}\mathcal{L}_{G}^{i}\middle|\overline{\left\langle \underline{g}\right\rangle }=xPx^{-1}\text{ for some }x\in G\right\} 
\]
is null for each $P\in\mathcal{P}$.

Thus let $P\in\mathcal{P}$, and pick some lifts $p_{1},\dots,p_{n}$
of $h_{1},\dots,h_{n}$ to $P$. Take a submanifold $A\subseteq K$
such that $AK_{P}=K$ and $\dim A+\dim K_{P}=\dim K$. Observe that
since 
\[
AP=AK_{P}P=KP=G,
\]
all subgroups conjugated to $P$ are of the form $aPa^{-1}$ for some
$a\in A$, and hence all lifts of $(h_{1},\dots,h_{n})$ to a subgroup
conjugated to $P$ are of the form $(ak_{1}p_{1}a{}^{-1},\dots,ak_{n}p_{n}a{}^{-1})$
for some $k_{1},\dots,k_{n}\in K_{P}$ and $a\in A$.
In other words, they are all in the image of the map
\begin{align*}
A\times{K_{P}}^{n} & \to\mathcal{L}_{G}^{1}\times\cdots\times\mathcal{L}_{G}^{n}\\
(a,k_{1},\dots,k_{n}) & \to(ak_{1}p_{1}a{}^{-1},\dots,ak_{n}p_{n}a{}^{-1}).
\end{align*}
Since this is a smooth map and 
\begin{align*}
\dim(A\times{K_{P}}^{n}) & =\dim A+n\dim K_{P}=\dim K+(n-1)\dim K_{P}\\
 & <n\dim K=\dim\left(\mathcal{L}_{G}^{1}\times\cdots\times\mathcal{L}_{G}^{n}\right),
\end{align*}
its image is null by Sard's theorem \cite[Theorem 10.5]{Lie}. The
proof for Lie groups is completed.

Since every first-countable compact group is an inverse limit of a
sequence of Lie groups (see e.g.~\cite[4.7]{MZ}),
we now just need to show an inverse limit of a sequence of connectedly
super-Gasch{\"u}tz compact groups is connectedly super-Gasch{\"u}tz.

Let $G=\invlim\left(G_{m}\right)_{m\in\mathbb{N}}$ be an inverse
limit of connectedly super-Gasch{\"u}tz compact groups, let $f:G\to H$
be an epimorphism with a connected kernel and let $\underline{h}\in H^{n}$ be a generating
system such that $n\geqslant\tr G$. We need to prove almost every
lift of $\underline{h}$ to $G$ generates $G$.

Denote $K=\ker f$, $K_{m}=p_{m}\left(K\right)$ (where $p_{m}:G\to G_{m}$
is the projection). Then $H\cong\invlim\left(G_{m}/K_{m}\right)_{m\in\mathbb{N}}$
and modulo this isomorphism $f$ is defined by $(g_{m})_{m\in\mathbb{N}}\mapsto(g_{m}K_{m})_{m\in\mathbb{N}}$.
For every $m\in\mathbb{N}$ the map $G_{m}\to G_{m}/K_{m}$ is super-Gasch{\"u}tz
(since $K_{m}=p_{m}\left(K\right)$ is connected). Pick some lift
$\underline{g}$ of $\underline{h}$, and denote 
\[
\mathcal{L}=\{\underline{k}\in K^{n}|\overline{\langle\underline{g}\cdot\underline{k}\rangle}\neq G\},
\]
which we need to show is null in $K^{n}$. Since $\underline{g}_{m}\coloneqq\underline{p}_{m}(\underline{g})$
is a lift of a generating system of $G_{m}/K_{m}$ (namely the image
of $\underline{h}$), the set 
\[
\mathcal{L}_{m}\coloneqq\{\underline{k}_{m}\in{K_{m}}^{n}|\overline{\langle\underline{g}_{m}\cdot\underline{k}_{m}\rangle}\neq G_{m}\}
\]
is null, and therefore $\tilde{\mathcal{L}}_{m}\coloneqq(\underline{p}_{m}\res{K^{n}})^{-1}(\mathcal{L}_{m})$
is null in $K^{n}$. It follows that $\bigcup_{m\in\mathbb{N}}\tilde{\mathcal{L}}_{m}$
is null. Since the latter set contains $\mathcal{L}$, we are done.
\end{proof}

\section{An Example}

In this section we show that the assumption of first-countability (or equivalently, metrisability)
cannot be omitted.
\begin{example}
A torus of dimension $2^{\aleph_{0}}$ is not Gasch{\"u}tz.
\end{example}
Observe that by Kronecker's theorem $(\mathbb{R}/\mathbb{Z})^{2^{\aleph_{0}}}$
is of topological rank $1$, so it makes sense to ask whether it is
Gasch{\"u}tz. We will in fact prove that in a sense, $(\mathbb{R}/\mathbb{Z})^{2^{\aleph_{0}}}$
is as not Gasch{\"u}tz as a group can be.
\begin{prop}
If $I\subsetneq J$ are index sets with $|I|=|J|=2^{\aleph_{0}}$
then for every $n\in\mathbb{N}$ there are $h_{1},\dots,h_{n}\in(\mathbb{R}/\mathbb{Z})^{I}$
which generate $(\mathbb{R}/\mathbb{Z})^{I}$ but do not admit generating
lifts to $(\mathbb{R}/\mathbb{Z})^{J}$. In particular, $(\mathbb{R}/\mathbb{Z})^{2^{\aleph_{0}}}$
is not Gasch{\"u}tz.
\end{prop}
Before we get to the proof, let us recall Kronecker's theorem.
\begin{thm}[Kronecker's theorem]
\label{thm:Kronecker}For any index set $I$, the elements $(\alpha_{1}^{i}+\mathbb{Z})_{i\in I},\dots,(\alpha_{m}^{i}+\mathbb{Z})_{i\in I}$
in $(\mathbb{R}/\mathbb{Z})^{I}$ generate $(\mathbb{R}/\mathbb{Z})^{I}$
if and only if the following implication holds: if $i_{1},\dots,i_{n}\in I$
and $\lambda_{1},\dots,\lambda_{n}\in\mathbb{Q}$ are such that
\begin{gather*}
\lambda_{1}\alpha_{1}^{i_{1}}+\cdots+\lambda_{n}\alpha_{1}^{i_{n}}\in\mathbb{Q}\\
\vdots\\
\lambda_{1}\alpha_{m}^{i_{1}}+\cdots+\lambda_{n}\alpha_{m}^{i_{n}}\in\mathbb{Q}
\end{gather*}
then $\lambda_{j}=0$ for every $j=1,\dots,n$.
\end{thm}
The theorem was proved in \cite{Kro} for finite index sets $I$.
The general case follows easily; if $(\alpha_{1}^{i}+\mathbb{Z})_{i\in I},\dots,(\alpha_{m}^{i}+\mathbb{Z})_{i\in I}$
are elements in $(\mathbb{R}/\mathbb{Z})^{I}$ and $\varGamma$ is the group they abstractly
generate, then the condition of the theorem is satisfied if and only if for any finite subset $S\subseteq I$ the projection
of $\varGamma$ in $(\mathbb{R}/\mathbb{Z})^{S}$ is dense. By definition of the product topology, this is equivalent
to $\varGamma$ being dense.
\begin{proof}[Proof of the Proposition]
Let $n\in\mathbb{N}$. Take disjoint $I_{1},\dots,I_{n}\subseteq I$
such that $\bigcup_{\ell=1}^{n}I_{\ell}=I$ and $\left|I_{1}\right|=\cdots=\left|I_{n}\right|=2^{\aleph_{0}}$.
Take 
\[
h_{1}=(h_{1}^{i}+\mathbb{Z})_{i\in I},\dots,h_{n}=(h_{n}^{i}+\mathbb{Z})_{i\in I}\in(\mathbb{R}/\mathbb{Z})^{I}
\]
such that for each $\ell=1,\dots,n$ the tuple $(h_{\ell}^{i})_{i\in I_{\ell}\cup\{0\}}$, where $h_0\coloneqq 1$,
is a basis of $\mathbb{R}$ over $\mathbb{Q}$, and $h_{\ell}^{i}=0$
for all $i\notin I_{\ell}\cup\{0\}$. By Kronecker's theorem, $h_{1},\dots,h_{n}\in(\mathbb{R}/\mathbb{Z})^{I}$
generate $(\mathbb{R}/\mathbb{Z})^{I}$. However, by the same theorem
the following computation shows no lifts of $h_{1},\dots,h_{n}$ to
$G$ could ever generate it.

Let $g_{1},\dots,g_{n}\in(\mathbb{R}/\mathbb{Z})^{J}$ be lifts of
$h_{1},\dots,h_{n}\in(\mathbb{R}/\mathbb{Z})^{I}$, and pick $j_{0}\in J\sm I$.
Then for each $\ell=1,\dots,n$ there are $i_{1}^{\ell},\dots,i_{d_{\ell}}^{\ell}\in I_{\ell}$
and $\lambda^{\ell},\lambda_{1}^{\ell},\dots,\lambda_{d_{\ell}}^{\ell}\in\mathbb{Q}$
such that 
\[
\left(\sum_{j=1}^{d_{\ell}}\lambda_{j}^{\ell}h_{\ell}^{i_{j}^{\ell}}\right)=\lambda^{\ell}-g_{\ell}^{j_{0}}
\]
(since $\{1\}\cup\{h_{\ell}^{i}\}_{i\in I_{\ell}}$ is a basis of
$\mathbb{R}$ over $\mathbb{Q}$ for each $\ell$). Since
for each $\ell=1,\dots,n$ we have $g_{\ell}^{i}=h_{\ell}^{i}=0$
for every $i\in I\sm I_{\ell}$, we will have for each $\ell=1,\dots,n$
and $k\neq\ell$ that 
\[
\sum_{j=1}^{d_{k}}\lambda_{j}^{k}g_{\ell}^{i_{j}^{k}}=0.
\]
Therefore we have for each $\ell=1,\dots,n$: 
\begin{align*}
\left(\sum_{k=1}^{n}\sum_{j=1}^{d_{k}}\lambda_{j}^{k}g_{\ell}^{i_{j}^{k}}\right)+g_{\ell}^{j_{0}} & =\left(\sum_{j=1}^{d_{\ell}}\lambda_{j}^{\ell}g_{\ell}^{i_{j}^{\ell}}\right)+g_{\ell}^{j_{0}}\\
 & =\left(\lambda^{\ell}-g_{\ell}^{j_{0}}\right)+g_{\ell}^{j_{0}}=\lambda^{\ell}\in\mathbb{Q}.
\end{align*}
Thus, by Kronecker's theorem, $g_{1},\dots,g_{n}$ do not generate
$(\mathbb{R}/\mathbb{Z})^{J}$.
\end{proof}

\begin{rem}
This result suggests the following definition: the \emph{Gasch{\"u}tz
rank }of a group $G$, denoted by $\gas G$, is the minimal $m$ such
that every open epimorphism $f:G\to H$ and every finite tuple of generators ($h_{1},\dots,h_{n})\in H^n$
with $n\geqslant m$ admit a generating lift $(g_{1},\dots,g_{n})\in G$.
Then $(\mathbb{R}/\mathbb{Z})^{2^{\aleph_{0}}}$ is not only not Gasch{\"u}tz,
but is of infinite Gasch{\"u}tz rank. It is not too hard to show $\gas{\mathbb{R}^{n}\times\mathbb{T}^{m}}=2n+m$ if $n>0$,
so if in addition $n\geqslant 2$ or $m\geqslant1$ then $\mathbb{R}^{n}\times\mathbb{T}^{m}$
is not Gasch{\"u}tz (because $\tr{\mathbb{R}^{n}\times\mathbb{T}^{m}} = n+1$) but is still of finite Gasch{\"u}tz rank.
\end{rem}

\appendix

\end{document}